\documentclass{elsart}
\usepackage{amsmath,amssymb,amsbsy,amsfonts,latexsym,
amsopn,amstext,amsxtra,euscript,amscd}

\usepackage{amssymb}

\begin{document}


\newfont{\teneufm}{eufm10}
\newfont{\seveneufm}{eufm7}
\newfont{\fiveeufm}{eufm5}
%
%
\newfam\eufmfam
 \textfont\eufmfam=\teneufm \scriptfont\eufmfam=\seveneufm
 \scriptscriptfont\eufmfam=\fiveeufm
%
%
\def\frak#1{{\fam\eufmfam\relax#1}}
%


\def\bbbr{{\rm I\!R}} 
\def\bbbm{{\rm I\!M}}
\def\bbbn{{\rm I\!N}} 
\def\bbbf{{\rm I\!F}}
\def\bbbh{{\rm I\!H}}
\def\bbbk{{\rm I\!K}}
\def\bbbp{{\rm I\!P}}
\def\bbbone{{\mathchoice {\rm 1\mskip-4mu l} {\rm 1\mskip-4mu l}
{\rm 1\mskip-4.5mu l} {\rm 1\mskip-5mu l}}}
\def\bbbc{{\mathchoice {\setbox0=\hbox{$\displaystyle\rm C$}\hbox{\hbox
to0pt{\kern0.4\wd0\vrule height0.9\ht0\hss}\box0}}
{\setbox0=\hbox{$\textstyle\rm C$}\hbox{\hbox
to0pt{\kern0.4\wd0\vrule height0.9\ht0\hss}\box0}}
{\setbox0=\hbox{$\scriptstyle\rm C$}\hbox{\hbox
to0pt{\kern0.4\wd0\vrule height0.9\ht0\hss}\box0}}
{\setbox0=\hbox{$\scriptscriptstyle\rm C$}\hbox{\hbox
to0pt{\kern0.4\wd0\vrule height0.9\ht0\hss}\box0}}}}
\def\bbbq{{\mathchoice {\setbox0=\hbox{$\displaystyle\rm
Q$}\hbox{\raise
0.15\ht0\hbox to0pt{\kern0.4\wd0\vrule height0.8\ht0\hss}\box0}}
{\setbox0=\hbox{$\textstyle\rm Q$}\hbox{\raise
0.15\ht0\hbox to0pt{\kern0.4\wd0\vrule height0.8\ht0\hss}\box0}}
{\setbox0=\hbox{$\scriptstyle\rm Q$}\hbox{\raise
0.15\ht0\hbox to0pt{\kern0.4\wd0\vrule height0.7\ht0\hss}\box0}}
{\setbox0=\hbox{$\scriptscriptstyle\rm Q$}\hbox{\raise
0.15\ht0\hbox to0pt{\kern0.4\wd0\vrule height0.7\ht0\hss}\box0}}}}
\def\bbbt{{\mathchoice {\setbox0=\hbox{$\displaystyle\rm
T$}\hbox{\hbox to0pt{\kern0.3\wd0\vrule height0.9\ht0\hss}\box0}}
{\setbox0=\hbox{$\textstyle\rm T$}\hbox{\hbox
to0pt{\kern0.3\wd0\vrule height0.9\ht0\hss}\box0}}
{\setbox0=\hbox{$\scriptstyle\rm T$}\hbox{\hbox
to0pt{\kern0.3\wd0\vrule height0.9\ht0\hss}\box0}}
{\setbox0=\hbox{$\scriptscriptstyle\rm T$}\hbox{\hbox
to0pt{\kern0.3\wd0\vrule height0.9\ht0\hss}\box0}}}}
\def\bbbs{{\mathchoice
{\setbox0=\hbox{$\displaystyle \rm S$}\hbox{\raise0.5\ht0\hbox
to0pt{\kern0.35\wd0\vrule height0.45\ht0\hss}\hbox
to0pt{\kern0.55\wd0\vrule height0.5\ht0\hss}\box0}}
{\setbox0=\hbox{$\textstyle  \rm S$}\hbox{\raise0.5\ht0\hbox
to0pt{\kern0.35\wd0\vrule height0.45\ht0\hss}\hbox
to0pt{\kern0.55\wd0\vrule height0.5\ht0\hss}\box0}}
{\setbox0=\hbox{$\scriptstyle  \rm S$}\hbox{\raise0.5\ht0\hbox
to0pt{\kern0.35\wd0\vrule height0.45\ht0\hss}\raise0.05\ht0\hbox
to0pt{\kern0.5\wd0\vrule height0.45\ht0\hss}\box0}}
{\setbox0=\hbox{$\scriptscriptstyle\rm S$}\hbox{\raise0.5\ht0\hbox
to0pt{\kern0.4\wd0\vrule height0.45\ht0\hss}\raise0.05\ht0\hbox
to0pt{\kern0.55\wd0\vrule height0.45\ht0\hss}\box0}}}}
\def\bbbz{{\mathchoice {\hbox{$\sf\textstyle Z\kern-0.4em Z$}}
{\hbox{$\sf\textstyle Z\kern-0.4em Z$}}
{\hbox{$\sf\scriptstyle Z\kern-0.3em Z$}}
{\hbox{$\sf\scriptscriptstyle Z\kern-0.2em Z$}}}}
\def\ts{\thinspace}

\newtheorem{theorem}{Theorem}
\newtheorem{lemma}[theorem]{Lemma}
\newtheorem{definition}{Definition}
\newtheorem{question}[theorem]{Open Question}

\newcommand{\algol}[2]{\ensuremath{\mathsf{#1}(#2)}}
\newcommand{\aln}[1]{\ensuremath{\mathsf{#1}}}
\newcommand{\ex}[2]{\ensuremath{\mathbf{#1}(#2)}}
\newcommand{\exn}[1]{\ensuremath{\mathbf{#1}}}
\newcommand{\scm}[2]{\ensuremath{\mathcal{#1}(#2)}}
\newcommand{\scn}[1]{\ensuremath{\mathcal{#1}}}

\newenvironment{proof}{\noindent{\it Proof. }}{{\qed}}

\def\squareforqed{\hbox{\rlap{$\sqcap$}$\sqcup$}}
\def\qed{\ifmmode\squareforqed\else{\unskip\nobreak\hfil
\penalty50\hskip1em\null\nobreak\hfil\squareforqed
\parfillskip=0pt\finalhyphendemerits=0\endgraf}\fi}

\def\cA{{\mathcal A}}
\def\cB{{\mathcal B}}
\def\cC{{\mathcal C}}
\def\cD{{\mathcal D}}
\def\cE{{\mathcal E}}
\def\cF{{\mathcal F}}
\def\cG{{\mathcal G}}
\def\cH{{\mathcal H}}
\def\cI{{\mathcal I}}
\def\cJ{{\mathcal J}}
\def\cK{{\mathcal K}}
\def\cL{{\mathcal L}}
\def\cM{{\mathcal M}}
\def\cN{{\mathcal N}}
\def\cO{{\mathcal O}}
\def\cP{{\mathcal P}}
\def\cQ{{\mathcal Q}}
\def\cR{{\mathcal R}}
\def\cS{{\mathcal S}}
\def\cT{{\mathcal T}}
\def\cU{{\mathcal U}}
\def\cV{{\mathcal V}}
\def\cW{{\mathcal W}}
\def\cX{{\mathcal X}}
\def\cY{{\mathcal Y}}
\def\cZ{{\mathcal Z}}
\newcommand{\rmod}[1]{\: \mbox{mod} \: #1}

\def\Tr{{\mathrm{Tr}}}

\def\mand{\qquad \mbox{and} \qquad}

\def\eqref#1{(\ref{#1})}


\newcommand{\ignore}[1]{}

\newcommand{\comm}[1]{\marginpar{%
\vskip-\baselineskip 
\raggedright\footnotesize
\itshape\hrule\smallskip#1\par\smallskip\hrule}}

\hyphenation{re-pub-lished}

\def\lln{{\mathrm Lnln}}
\def\ad{{\mathrm ad}}
\def\scr{\scriptstyle}

\def \C{{\bbbc}}
\def \F{{\bbbf}}
\def \K{{\bbbk}}
\def \Z{{\bbbz}}
\def \N{{\bbbn}}
\def \Q{{\bbbq}}
\def \R{{\bbbr}}
\def\Fp{\F_p}
\def \fp{\Fp^*}
\def \Zm{\Z_m}
\def \Um{{\mathcal U}_m}

\def\Km{\cK_\mu}

\def \va {{\mathbf a}}
\def \vb {{\mathbf b}}
\def \vc {{\mathbf c}}
\def \vx{{\mathbf x}}
\def \vr {{\mathbf r}}
\def \vv {{\mathbf v}}
\def \vu{{\mathbf u}}
\def \vw{{\mathbf w}}
\def \vz {{\mathbf z}}

\def\\{\cr}
\def\({\left(}
\def\){\right)}
\def\fl#1{\left\lfloor#1\right\rfloor}
\def\rf#1{\left\lceil#1\right\rceil}

\def\flq#1{{\left\lfloor#1\right\rfloor}_q}
\def\flp#1{{\left\lfloor#1\right\rfloor}_p}

\def\Al{{\sl Alice}}
\def\Bob{{\sl Bob}}

\def\Or{{\mathcal O}}

\def\invM#1{\mbox{\rm {inv}}_M\,#1}
\def\invp#1{\mbox{\rm {inv}}_p\,#1}

\def\Ln#1{\mbox{\rm {Ln}}\,#1}

\def \nd {\, | \hspace{-1.2mm}/\,}

\def\ord{\mu}

 \def\e{\mbox{\bf{e}}}

\def\ep{\mbox{\bf{e}}_p}
\def\eq{\mbox{\bf{e}}_q}
\newcommand{\floor}[1]{\lfloor {#1} \rfloor }

\def\rem{{\mathrm{\,rem\,}}}
\def\dist {{\mathrm{\,dist\,}}}


\begin{frontmatter}

\title{On RSA Moduli with Almost Half of the Bits Prescribed}

\author{Sidney W. Graham}

\address{{Department of Mathematics, Central Michigan University}\\
 Mount Pleasant, MI 48859,
USA\\ {\tt sidney.w.graham@cmich.edu}}
%

\author{Igor E. Shparlinski}

\address{{Department of Computing, 
Macquarie University}\\ {Sydney, NSW 2109, Australia}\\    {\tt 
igor@ics.mq.edu.au}
}

\date{\today}


\begin{abstract}
We show that using character sum estimates due to H.~Iwaniec
leads to an improvement of recent results
about the distribution and finding  RSA
moduli $M=pl$, where $p$ and $l$ are primes,
with prescribed bit patterns. We are now able to specify
about $n$ bits  instead of about $n/2$ bits as in the
previous work.
We also show that the same result of H.~Iwaniec
can be used to obtain an unconditional
version of a combinatorial result of W.~de~Launey and D.~Gordon
that was originally  derived under the Extended Riemann Hypothesis.
\end{abstract}


\end{frontmatter}

\newcommand{\zo }[0]{\{0,1\} }
\newcommand{\ind}[0]{\stackrel{\mbox{\tiny c}}{\approx}}
\newcommand{\p}[0]{\mbox{\rm Prob}}

\section{Introduction}

For an integer $n$, we use $\cP_n$ to denote the set of
primes $p$ with $2^{n-1} < p< 2^n$.
Let $\cM_n$ be the set of RSA moduli $M = p \ell$
that are products of two distinct primes $p,  \ell \in \cP_n$.

Thus each $M\in \cM_n$ has either $2n-1$ or $2n$ bits
which we number from the right to the left.

Motivated by some cryptographic applications (in particular 
by the idea of reducing the size of the public key), various heuristic
algorithms have been given to construct moduli $M \in \cM_n$ having a
sufficiently long specified  bit pattern have been
given in~\cite{LenA,VaZu}. Unfortunately, giving a rigorous analysis
of these algorithms require a very strong form of {\it Linnik's
Theorem\/}, which far exceeds our current state of knowledge.

A different algorithm was  proposed in~\cite{Shp}.
Certainly this algorithm is likely to produce moduli having shorter
prescribed bit patterns than those of~\cite{LenA,VaZu}.
However, using exponential sums,
this algorithm has been rigorously analysed and shown to
output in expected polynomial time
a desired modulus $M \in \cM_n$ with about $n/2$ prescribed bits.

Here we use the bound of character sums of H.~Iwaniec~\cite{Iwan}
(see also~\cite{FuGaMo,Gal} and references therein)
instead of bounds on exponential sums.
This allows us to show that in fact the same algorithm can be used to generate
in expected polynomial time
a desired RSA modulus  $M \in \cM_n$  with about $n$ prescribed bits.

Our result immediately yields
an improvement Theorem~5 in~\cite{Shp}, producing RSA
moduli $M\in \cM_n$ with with at least
$(3/2 + o(1))n$ zero bits.
As in~\cite{Shp} we remark that such moduli may be useful
for the {\it Paillier cryptosystem\/}, see~\cite{Pail},
where one computes $M$th powers.

We also outline some possible applications of the same ideas to generating
sparse RSA moduli and  smooth numbers (that is, numbers
free of large prime factors) with a prescribed bit pattern,
hence improving some other results of~\cite{Shp}.

We end up with an observation that the results
of~\cite{Iwan} can also be used to eliminate the
assumption of  the Extended Riemann Hypothesis
from a result of W.~de~Launey and D.~Gordon~\cite{dLaGor}

Throughout the paper, $\cP$ denotes the set of primes and $\ln z$ denotes the
natural logarithm of $z>0$.

\section{RSA Moduli with Prescribed Bit Patterns}
\label{sec:RSA-Prescr}

We recall the algorithm of~\cite{Shp} to generate an RSA modulus $M$
having a desired bit pattern on certain positions.

Given a binary string $\sigma$ of length $m$, we denote by
$\cM_{n,m}(\sigma)$ the set consisting  of $M \in \cM_n$
such that the bits of $M$ at the positions $n-1, \ldots, n-m$
form the binary string  $\sigma$.

\begin{center}
\ \\
\bf Algorithm~\algol{RSA-Modulus}{n,m,\sigma}
\end{center}

\begin{enumerate}
\item[Step~1]  Choose an odd
integer $k$ in the interval $1 \le k < 2^{n-m}$
and a  prime $p \in \cP_n$ uniformly at random.

\item[Step~2]  Compute the positive integer $r < 2^n$
which satisfies the congruence
$$ pr  \equiv 2^{n-m} s + k \pmod{2^n},
$$
where $s$ is the integer whose binary representation coincides with $\sigma$.

\item[Step~3] Test whether
$2^{n-1} < r$, $r \ne p$ and also test $r$ for primality,
if $r$ is prime then put $l = r$ and output
$M = pl$, otherwise go to Step~1 and start a new round
of the algorithm.
\end{enumerate}

Certainly, if  Algorithm~{\sc RSA-Modulus}$(n,m,\sigma)$ terminates it outputs
$M \in \cM_{n,m}( \sigma)$.

\begin{theorem}
\label{thm:PrescrRSA}
For  $m =  \fl{n  - n^{3/4} \ln n}$ and any binary string $\sigma$
of length $m$, Algorithm~{\sc RSA-Modulus}$(n,m,\sigma)$
terminates in expected  polynomial time.
\end{theorem}

\begin{proof} As in~\cite{Shp}, for an
integer  $0 \le k \le 2^{n-m}-1$, we denote by
$N(k)$ the number of solutions $p,l \in \cP_n$ to the
congruence $ p l  \equiv 2^{n-m} s +k\pmod{2^n}$
where binary representation of
the integer $s$ is given by the  string $\sigma$
(certainly $N(k) = 0$ for every even $k$).

Let $\cX$ be the set of multiplicative characters modulo $2^n$
(see~\cite[Chapter~3]{IwKow} for a background on characters and 
character sums).

We also use $\cX^*$ to denote the set of nonprincipal characters.
We recall the orthogonality relation
\begin{equation}
\label{eq:ident}
\sum_{\chi\in \cX} \chi( u) =
\left\{ \begin{array}{ll}
0,& \quad \mbox{if}\ u\not \equiv 1 \pmod {2^n}, \\
2^{n-1},& \quad \mbox{if}\ u \equiv 1 \pmod {2^n},
\end{array} \right.
\end{equation}
see~\cite[Section~3.2]{IwKow}.

By~\eqref{eq:ident}, we have
$$
N(k)= \frac{1}{2^{n-1}}
\sum_{p, \ell \in \cP_n}
\sum_{\chi\in \cX} \chi\((2^{n-m} s -  k)p^{-1}\ell^{-1} \),
$$
where the inverse values $p^{-1}$ and $\ell^{-1}$ are taken
modulo $2^n$.

Changing the order of summation and separating the term $(\#\cP_n)^2 2^{-n+1}$
corresponding to the principal character we obtain
\begin{eqnarray*}
N(k) & = &(\#\cP_n)^2 2^{-n+1}  +
\frac{1}{2^{n-1}}\sum_{\chi\in \cX^*} \chi\(2^{n-m} s  + k\)
\sum_{p,l \in \cP_n} \chi\( p^{-1} \ell^{-1} \)\\
& = &(\#\cP_n)^2 2^{-n+1}  +
\frac{1}{2^{n-1}}\sum_{\chi\in \cX^*} \chi\(2^{n-m} s  + k\)
   \(\sum_{p  \in \cP_n}  \chi\(p^{-1}\)\)^2.
\end{eqnarray*}
Therefore
\begin{equation}
\label{eq:aver N(k)}
\sum_{k=0}^{2^{n-m}-1} N(k) = \frac{(\#\cP_n)^2} {2^{m}}  +  \Delta,
\end{equation}
where
$$
\Delta = \frac{1}{2^{n-1}} \sum_{\chi\in \cX^*}
\sum_{k=0}^{2^{n-m}-1}
\chi\(2^{n-m} s  + k\)  \(\sum_{p  \in \cP_n} \chi\(p^{-1}\)\)^2.
$$
Using the triangle inequality, we conclude that
\begin{eqnarray*}
|\Delta| & = & \frac{1}{2^{n-1}} \left| \sum_{\chi\in \cX^*}
\sum_{k=0}^{2^{n-m}-1}
\chi\(2^{n-m} s  + k\)  \(\sum_{p  \in \cP_n} \chi\(p^{-1}\)\)^2\right|\\
& \le & \frac{1}{2^{n-1}}\sum_{\chi\in \cX^*}\left|  \sum_{k=0}^{2^{n-m}-1}
\chi\(2^{n-m} s  + k\) \right| \left|  \sum_{p  \in \cP_n} \chi\(p^{-1}\)\right|^2\\
& = & \frac{1}{2^{n-1}}\sum_{\chi\in \cX^*}\left|  \sum_{k=0}^{2^{n-m}-1}
\chi\(2^{n-m} s  + k\) \right| \left|  \sum_{p  \in \cP_n} \chi\(p\)\right|^2.
\end{eqnarray*}
since the values of $\chi(p)$ and $\chi\(p^{-1}\)$ are conjugated over $\C$.

We now recall that by~\cite[Lemma 6]{Iwan},
$$
   \sum_{k=0}^{2^{n-m}-1}
\chi\(2^{n-m} s  + k\)  \ll 2^{n-m} n^{-2}
$$
provided that
$$
2^{n-m} \gg 2^{n^{3/4} \ln n}.
$$
which is satisfied for our choice of $m$.

Therefore
\begin{eqnarray*}
|\Delta| & \ll  & \frac{1}{2^{m} n^2 } \sum_{\chi\in \cX^*} \left|
\sum_{p  \in \cP_n}
\chi\(p\)\right|^2 \le \frac{1}{2^{m} n^2 } \sum_{\chi\in \cX} \left|
\sum_{p  \in \cP_n}
\chi\(p\)\right|^2\\
& =  & \frac{1}{2^{m} n^2 } \sum_{p, \ell  \in \cP_n}  \sum_{\chi\in \cX}
\chi\(p \ell^{-1}\)  .
\end{eqnarray*}
By~\eqref{eq:ident} we see that inner sum vanishes unless
$p \equiv \ell \pmod {2^n}$, which is equivalent to $p = \ell$.
Therefore
\begin{equation}
\label{eq:Bound Delta}
|\Delta| \ll \frac{2^{n-m} \#\cP_n}{ n^2 } .
\end{equation}

Since $\# \cP_n \gg 2^n n^{-1}$, substituting the bound~\eqref{eq:Bound Delta}
in~\eqref{eq:aver N(k)} we derive
$$
\sum_{k=0}^{2^{n-m}-1} N(k) =
\left(1 + O(n^{-1}) \right)\frac{(\#\cP_n)^2} {2^{m}}.
$$
This is a full analogue of the asymptotic formula~(4) in~\cite{Shp}
(except that the value of $m$ is now different). Accordingly, the
rest of the proof is
identical to that of Theorem~4 in~\cite{Shp}.
\end{proof}

As we have remarked, Theorem~\ref{thm:PrescrRSA} immediately yields
the following improvement Theorem~5 in~\cite{Shp} which can be
have some application for 
for the {\it Paillier cryptosystem\/} (see~\cite{Pail}).

\begin{cor}
\label{cor:SparseRSA}
For  $m =  \rf{n -  n^{3/4} \ln n}$ and the $m$-dimensional
zero string $\vartheta=(0, \ldots, 0)$ of length $m$,  Algorithm~{\sc
RSA-Modulus}$(n,m,\vartheta)$ terminates in expected  polynomial time 
and with probability $1
+ o(1)$ outputs a
modulus $M \in \cM_n$ with at most $(1/2 + o(1))n $ nonzero bits.
\end{cor}

\section{Other Applications}

One can use a similar approach to improve Theorem~6
of~\cite{Shp} which guarantees the existence of
certain smooth  numbers with prescribed bit
patterns.

Moreover, without any substantial changes,
an analogue of Theorem~\ref{thm:PrescrRSA} can be
obtained for the values of the Euler function
$\varphi(p \ell) = (p-1) (\ell -1)$. For example,
one can prove that for any $r$, there are $r$-bit integers
$R$ such that the binary expansion of
$\varphi(R)$ contains  $(3/4 + o(1))r$ nonzero digits.
This can be extended to $g$-ary expansions for any
base $g$.

Finally, we conclude with noticing that the 
results of~\cite{Gal} and~\cite{Iwan}
have direct implications on the distribution of primes in
arithmetic progressions modulo $2^n$. 
P.~X.~Gallagher~\cite{Gal}  
proves that if $q=p^r$ ($p$ odd)
and if $q\cdot x^{3/5+\epsilon} \le h \le x$, then 
\begin{equation}
\label{eq:PrimesInAP}
\psi(x+h,q,a)-\psi(x,q,a) \sim \frac{h}{\varphi(q)}
\end{equation}
whenever $(a,q)=1$, where, as usual, 
$$\psi(x,q,a) = \sum_{\substack{k \le x\\ k \equiv a \pmod q}} \Lambda(k)
$$
and
$$\Lambda(k)=
\begin{cases}
\log p&\qquad\text{if $k$ is a power of a prime $p$,} \\
0&\qquad\text{otherwise,}
\end{cases}$$
is the von Mangoldt function, see~\cite[Chapter~5.9]{IwKow}. 
The exponent $3/5$ came from appealing to a zero-density estimate of 
H.~L.~Montgomery~\cite{Mon}.
For technical reasons, P.~X.~Gallagher~\cite{Gal}  
excludes consideration of the case 
$p=2$, but his proof can be easily modified to this case.
The details of this modification (and much more) 
have been 
given by H.~Iwaniec~\cite{Iwan}.
By using a zero-density result of Huxley~\cite{Hux}
in conjunction with~\cite{Iwan}, one sees that 
\eqref{eq:PrimesInAP} is true with $q=2^r$ and 
$q\cdot x^{7/12+\epsilon} \le h \le x$. 
This result can be used in place of the Extended Riemann Hypothesis
in the paper of W.~de~Launey and D.~Gordon~\cite{dLaGor}. 
In particular, in the last undisplayed equation on page~184
of~\cite{dLaGor}, one may take 
$y=\fl{n^{7/12+\epsilon}}.$
In turn,  this yields an unconditional version of
Theorem~1.2 of~\cite{dLaGor}, albeit  with a weaker error term:
$$
r(N) \ge \frac{N}{2} + O(N^{113/132 + o(1)})
$$
for any $N \equiv 0 \pmod 4$, where $r(N)$ is the largest $R$
for which there is a $R\times N$ Hadamard matrix  (that is,
$\pm 1$ matrix $H$ with $HH^{T} = N I_R$,
where $I_R$ is the $R\times R$ identity matrix).
The exponent $113/132$ arises as
$$
    \frac{\alpha}{1+\alpha} + \frac{7}{12} = \frac{113}{132},
$$
where, as in~\cite{dLaGor}, we take $\alpha=3/8$. 
In the conditional result of 
W.~de~Launey and D.~Gordon~\cite{dLaGor},
the $7/12$ term is replaced by $1/2$, thus giving the 
exponent
$$\frac{7}{22} =\frac{113}{132}-\frac{1}{12}.$$

\end{document}